\documentclass[12pt]{amsart}
\usepackage{amssymb,latexsym,comment,url}
\usepackage[all]{xy}

\newfont{\wncyr}{wncyr10 at 12pt}
\newfont{\wncyrten}{wncyr10 at 10pt}

\newcommand{\adj}{{\operatorname{adj}}}
\newcommand{\tr}{{\operatorname{tr}}}

\newcommand{\SL}{{\operatorname{SL}}}

\newcommand{\charic}{\operatorname{char}}

\newcommand{\Mat}{\operatorname{Mat}}
\newcommand{\Div}{\operatorname{Div}}

\newcommand{\GGG}{{\mathcal G}}

\newcommand{\CC}{{\mathcal C}}
\newcommand{\GG}{{\widetilde{G}}}
\newcommand{\OK}{{\mathcal O}_K}

\newcommand{\PP}{{\mathbb P}}

\newcommand{\Q}{{\mathbb Q}}
\newcommand{\C}{{\mathbb C}}

\newcommand{\Z}{{\mathbb Z}}
\newcommand{\W}{{\mathcal W}}

\newcommand{\GL}{{\operatorname{GL}}}

\newenvironment{ProofOf}[1]{\par\noindent{\em Proof of #1.}}%
                        {\hspace*{\fill}\nobreak$\Box$\par\medskip}

\newtheorem{Proposition}{Proposition}[section]
\newtheorem{thm}[Proposition]{Theorem}
\newtheorem{lemma}[Proposition]{Lemma}
\newtheorem{Corollary}[Proposition]{Corollary}

\theoremstyle{definition}

\newtheorem{Remark}[Proposition]{Remark}

\begin{document}

\title[Minimisation algorithms]%
{Some minimisation algorithms \\ in arithmetic invariant theory}

\author{Tom~Fisher}
\address{University of Cambridge,
         DPMMS, Centre for Mathematical Sciences,
         Wilberforce Road, Cambridge CB3 0WB, UK}
\email{T.A.Fisher@dpmms.cam.ac.uk}
\author{Lazar~Radi\v{c}evi\'{c}}
\address{}
\email{}

\date{6th March 2017}  

\begin{abstract}
  We extend the work of Cremona, Fisher and Stoll on minimising genus
  one curves of degrees $2,3,4,5$, to some of the other
  representations associated to genus one curves, as studied by
  Bhargava and Ho.  Specifically we describe algorithms for minimising
  bidegree $(2,2)$-forms, $3 \times 3 \times 3$ cubes and $2 \times 2
  \times 2 \times 2$ hypercubes.  We also prove a theorem relating the
  minimal discriminant to that of the Jacobian elliptic curve.
\end{abstract}

\maketitle

\renewcommand{\baselinestretch}{1.1}
\renewcommand{\arraystretch}{1.3}
\renewcommand{\theenumi}{\roman{enumi}}

\section{Introduction}

Let $F$ be a homogeneous polynomial in several variables with rational
coefficients. Then making a linear change of variables and rescaling
the polynomial by a rational number does not change the isomorphism
class of the hypersurface defined by $F$. Thus a natural question is
to find a change of variables and a rescaling of the polynomial so
that its coefficients are small integers.

More generally we may consider the following situation.
Let $\GGG$ be a product of general linear groups, acting linearly on a
$\Q$-vector space $W$. We fix a basis for $W$, and represent a vector
$w \in W$ by its vector of co-ordinates $(w_1,\ldots,w_N)$ relative to
this basis.  We refer to these co-ordinates as the {\em coefficients}.
Then given $w \in W$ we seek to find $g \in \GGG(\Q)$ such that $g
\cdot w$ has small integer coefficients.

An {\em invariant} is a polynomial $I \in \Z[w_1,\ldots,w_N]$ such
that:
\[
    I(g \cdot w)=\chi(g) I(w)
\]
for all $g \in \GGG(\C)$ and $w \in W$, where $\chi$
is a rational character on $\GGG$ (i.e. a product of determinants).
In practice there will be 
an invariant $\Delta$, which we call the {\em discriminant}, and the
elements $w \in W$ of interest will be those with $\Delta(w) \not =
0$.  We note that if $w$ has integer coefficients then $\Delta(w)$ is
an integer.  Our strategy is to first find $g \in \GGG(\Q)$ making
this discriminant as small as possible (in absolute value).  This is
known as $minimisation$. This is a local problem, in that for each
prime $p$ dividing $\Delta(w)$ we seek to minimise the $p$-adic
valuation $v_p(\Delta(w))$, without changing the valuations at the
other primes.  Once we've minimised the discriminant, the next step is
to find a transformation in $\GGG(\Z)$, making the coefficients as
small as possible. This is known as $reduction$.

This strategy has been carried out in \cite{CFS} and \cite{F5}, for
the models (i.e. collections of polynomials) defining genus one curves
of degrees $2,3,4$ and $5$.  In these cases the invariants give a
Weierstrass equation for the Jacobian of the genus one curve. In this
article, we extend these techniques to some of the other representation
associated to genus one curves, as studied in \cite{BH}.  Specifically
we describe algorithms for minimising bidegree $(2,2)$-forms, $3
\times 3 \times 3$ cubes and $2 \times 2 \times 2 \times 2$
hypercubes. In each of these cases the invariants define not only the
Jacobian elliptic curve $E$, but also one or two marked points on
$E$. One possible application of these algorithms is in computing the
Cassels-Tate pairing (see~\cite{CTP3}).

As explained below, each $(2,2)$-form $F$ determines a pair of binary
quartics $G_1, G_2$, each $3 \times 3 \times 3$ cube $S$ determines a
triple of ternary cubics $F_1, F_2, F_3$, and each $2 \times 2 \times
2 \times 2$ hypercube $H$ determines a quadruple of binary quartics
$G_1, \ldots, G_4$. Therefore a natural approach would be to minimise
and reduce the corresponding binary quartics and ternary cubics, using
the algorithms in \cite{CFS}, and then apply the transformations that
arise in this way to $F$, $S$ or $H$.  This strategy works for
reduction (which we therefore do not study further in this article),
but not for minimisation.
For example if $F \in \Z[x_1,x_2;y_1,y_2]$ is a $(2,2)$-form with $F
\equiv x_2^2 y_2^2 \pmod{p^2}$ then the binary quartics $G_1$ and
$G_2$ vanish mod $p^2$. The algorithm for minimising binary quartics
says that we should divide each $G_i$ by $p^2$.
However 
this information on its own does not tell us how to minimise~$F$.

Since minimisation is a local problem, we work in the following
setting. Let $K$ be a field with a discrete valuation $v : K^\times
\to \Z$.  We write $\OK$ for the valuation ring, and $\pi$ for a
uniformiser, i.e. an element $\pi \in K$ with $v(\pi)=1$.  The residue
field is $k = \OK/\pi \OK$. For example we could take $K = \Q$ or
$\Q_p$, and $v=v_p$ the $p$-adic valuation. In these cases $\OK =
\Z_{(p)}$ or $\Z_p$. We make no restrictions on the characteristics
of $K$ and $k$.

Since it serves as a prototype for our work, we briefly recall the
algorithm for minimising binary quartics.  See \cite{CFS} for further
details.  A binary quartic is a homogeneous polynomial of degree $4$
in two variables:
\[ G(x_1,x_2) = a x_1^4 + b x_1^3 x_2 + c x_1^2 x_2^2 + d x_1 x_2^3 +
e x_2^4. \] If $R$ is any ring then there is an action of $\GGG(R) =
R^\times \times \GL_2(R)$ on the space of binary quartics over $R$ via
\begin{equation}
\label{act:bq}
 \left[ \lambda,\begin{pmatrix} r & s  \\ t & u  \end{pmatrix} \right]: 
  G(x_1,x_2) \mapsto \lambda^2 G(rx_1+tx_2,sx_1+ux_2). 
\end{equation}
We say that binary quartics are {\em $R$-equivalent} if they belong to
the same orbit for this action.  A polynomial $I \in \Z[a,b,c,d,e]$ is
an invariant of weight $p$ if
\[ I([\lambda,A] \cdot G) = (\lambda \det A )^p I(G) \]
for all $[\lambda,A] \in \GGG(\C)$. The invariants of a binary
quartic are
\begin{align*}
I & = 12 a e - 3 b d + c^2 \\
J & = 72ace - 27ad^2 - 27b^2e + 9bcd - 2c^3 
\end{align*} 
of weights $4$ and $6$, and $\Delta = (4 I^3 - J^2)/27$ of weight $12$.

A binary quartic $G$ is {\em integral} if it has coefficients in
$\OK$, and {\em non-singular} if $\Delta(G) \not=0$.  We write $v(G)$
for the minimum of the valuations of the coefficients of~$G$.  Given a
non-singular binary quartic, we seek to find a $K$-equivalent integral
binary quartic $G$ with $v(\Delta(G))$ as small as possible.

We write $\GG$ for the reduction of $\pi^{-v(G)} G$ mod $\pi$.  If a
binary quartic $G(x_1,x_2)$ is non-minimal, then it is
$\OK$-equivalent to a binary quartic with
\[ G(x_1 , \pi^s x_2) \equiv 0 \pmod{\pi^{2s+2}} \] for some integer
$s \ge 0$. The least such integer $s$ is called the {\em slope}, and
can only take values $0,1$ and $2$.  If $v(G) \le 1$ (i.e. the slope
is positive) then $\GG$ has a unique multiple root, and if we move
this root to $(1:0)$ then $\pi^{-2} G(x_1, \pi x_2)$ is an integral
binary quartic with the same invariants, but with smaller slope.
After at most two iterations we reach a form of slope $0$.  We can
then divide through by $\pi^2$, and repeat the process until a minimal
binary quartic is obtained.

Our algorithms for minimising $(2,2)$-forms, $3 \times 3 \times 3$
cubes and $2 \times 2 \times 2 \times 2$ hypercubes are described in
Sections~\ref{sec:22}, \ref{sec:cubes} and \ref{sec:hypercubes}.  We
also give formulae for the Jacobian elliptic curve and the marked
points that work in all characteristics. (In \cite{BH} the authors
worked over a field of characteristic not $2$ or $3$, and the formulae
were not always given explicitly.)  In Section~\ref{sec:minthm} we
prove a theorem about the minimal discriminant, and describe how it is
improved by our minimisation algorithms.

\section{Bidegree (2,2)-forms}
\label{sec:22}

A {\em $(2,2)$-form} is a polynomial in $x_1,x_2, y_1,y_2$, that is
homogeneous of degree $2$ in both sets of variables.  We can view a
$(2,2)$-form $F$ as a binary quadratic form in $y_1, y_2$ whose
coefficients are binary quadratic forms in $x_1, x_2$:
\begin{equation*}
  F(x_1,x_2;y_1,y_2) 
  = F_1(x_1,x_2) y_1^2 + F_2(x_1,x_2) y_1 y_2 + F_3(x_1,x_2) y_2^2.
\end{equation*}
The discriminant $G_1 = F_2^2 - 4 F_1 F_3$ is then a binary quartic in
$x_1,x_2$. Switching the two sets of variables we may likewise define
a binary quartic $G_2$ in $y_1,y_2$.
It may be checked that $G_1$ and $G_2$ have the same invariants $I$
and $J$. We define $c_4(F) = I$ and $c_6(F) = J/2$. The discriminant
is $\Delta(F) = (c_4^3 - c_6^2)/1728$.

A non-zero $(2,2)$-form $F$ over a field defines a curve in
$\mathbb{P}^1 \times \mathbb{P}^1$. If $\Delta(F) \not=0$ then this
curve $\CC_F$ is a smooth curve of genus one.  It may be written as a
double cover of $\PP^1$ (ramified over the roots of $G_1$ or $G_2$) by
projecting to either factor.

Let $R$ be a ring. There is an action of $\GGG(R) = R^\times \times
\GL_2(R) \times \GL_2(R)$ on the space of (2,2)-forms over $R$ given
by
\[ [\lambda,A , B ] : F(x_1,x_2;y_1,y_2) \mapsto 
\lambda F( (x_1,x_2) A ; (y_1,y_2) B). \] We say that $(2,2)$-forms
are {\em $R$-equivalent} if they belong to the same orbit for this
action. If $[\lambda, A,B] \cdot F = F'$ then the binary quartics
$G_1$ and $G_2$ determined by $F$, and the binary quartics $G'_1$ and
$G'_2$ determined by $F'$, are related by
\begin{equation}
\label{22->2}
\begin{aligned}
G'_1 &=[\lambda \det B, A] \cdot G_1 \\
G'_2 &=[\lambda \det A, B] \cdot G_2 
\end{aligned}
\end{equation}
where the action on binary quartics is that defined in~\eqref{act:bq}.

We may represent $F$ by a $3 \times 3$ matrix via:
\begin{equation}
\label{mat:22}
F(x_1,x_2;y_1,y_2) = \begin{pmatrix} x_1^2 & x_1 x_2 & x_2^2 \end{pmatrix}
\begin{pmatrix} 
a_{11} & a_{12} & a_{13} \\
a_{21} & a_{22} & a_{23} \\
a_{31} & a_{32} & a_{33} 
\end{pmatrix} \begin{pmatrix} y_1^2 \\ y_1y_2 \\ y_2^2 \end{pmatrix}.
\end{equation}

A polynomial $I \in \Z[a_{ij}]$ is an {\em invariant} of weight $p$ if
\[
I([\lambda,A,B]\cdot F)=(\lambda \det A \det B)^{p} I(F)
\]
for all $[\lambda,A,B] \in \GGG(\C)$. In particular the polynomials
$c_4, c_6$ and $\Delta$ are invariants of weights $4$, $6$ and
$12$. Over a field of characteristic not $2$ or $3$, the invariants
determine a pair $(E,P)$ where $E$ is an elliptic curve (the Jacobian
of $\CC_F$) and $P$ is a marked point on $E$. The next lemma gives
formulae for
\[ E: \qquad y^2 + a_1 xy + a_3 y = x^3 + a_2 x^2 + a_4 x + a_6 \] and
$P=(\xi,\eta)$ that work in all characteristics.

\begin{lemma}
\label{lem:inv22}
There exist $\xi,\eta,a_1,a_2,a_3,a_4,a_6 \in \Z[a_{ij}]$ such that
\begin{enumerate}
\item We have $c_4 = b_2^2 - 24 b_4$ and $c_6 = -b_2^3 + 36 b_2 b_4 -
  216 b_6$, where $b_2 = a_1^2 + 4 a_2$, $b_4 = a_1 a_3 + 2 a_4$ and
  $b_6 = a_3^2 + 4 a_6$,
\item The polynomials $u = 12 \xi + a_1^2 + 4 a_2$ and $v = 2 \eta +
  a_1 \xi + a_3$ are invariants of weights $2$ and $3$ satisfying
  $(108 v)^2 = (3 u)^3 - 27 c_4 (3 u) - 54 c_6$.
\item We have $\eta^2 + a_1 \xi \eta + a_3 \eta = \xi^3 + a_2 \xi^2 +
  a_4 \xi + a_6$.
\end{enumerate}
\end{lemma}

\begin{proof}
  We put $\xi = a_{11} a_{33} + a_{13} a_{31}$ and $\eta = a_{11} a_{22} a_{33}$. \\
  (i) We put
\begin{align*}
  a_1 &= -a_{22}, \\
  a_2 &= -(a_{11} a_{33}+ a_{12} a_{32} + a_{13} a_{31} + a_{21} a_{23}), \\
  a_3 &= a_{12} a_{23} a_{31} + a_{13} a_{21} a_{32} - a_{11} a_{23}
  a_{32} - a_{12} a_{21} a_{33}.
\end{align*}
Since we already defined $c_4$ and $c_6$, we may solve for $a_4$ and
$a_6$. We find 
that these too are polynomials in the $a_{ij}$ with integer coefficients. \\
(ii) The invariants $u$ and $v$ were denoted $\delta_2$ and $\delta_3$
in \cite[Section 6.1.2]{BH}. In fact we have $v = \det(a_{ij})$. \\
(iii) This follows from (i) and (ii), exactly as in \cite[Chapter III]{Sil}.
\end{proof}

Let $(E,P)$ be a pair consisting of an elliptic curve $E/K$ and a
point $0_E \not= P \in E(K)$. On a minimal Weierstrass equation for
$E$, the point $P$ has co-ordinates $(x_P,y_P)$, where either $x_P,y_P
\in \OK$ or $v(x_P) = -2r$, $v(y_P) = -3r$ for some integer $r \ge
1$. We define $\kappa(P) = 0$ in the first case, and $\kappa(P) = r$
in the second. We write $\Delta_E$ for the minimal discriminant of
$E$.

We say that a $(2,2)$-form $F$ is {\em integral} if it has
coefficients in $\OK$, and {\em non-singular} if $\Delta(F) \not=0$.

\begin{lemma}
\label{lem:lev22}
Let $F$ be a non-singular integral $(2,2)$-form. Let $(E,P)$ be the
pair specified in Lemma~\ref{lem:inv22}. Then
\[   v(\Delta(F)) = v(\Delta_E) + 12 \kappa(P) + 12 \ell(F) \]
where $\ell(F) \ge 0$ is an integer we call the {\em level}.
\end{lemma}
\begin{proof}
  The formulae in Lemma~\ref{lem:inv22} give an integral Weierstrass
  equation $W$ for $E$, upon which $P$ is a point with integral
  coordinates.  The smallest possible discriminant of such an equation
  is $v(\Delta_E) + 12 \kappa(P)$. Since the discriminant of $F$ is
  equal to the discriminant of $W$, the result follows.
\end{proof}

In this section we give an algorithm for minimising $(2,2)$-forms.
That is, given a non-singular $(2,2)$-form $F$ over $K$, we explain 
how to find a $K$-equivalent integral $(2,2)$-form with level
(equivalently, valuation of the discriminant) as small as possible. In
Section~\ref{sec:minthm} we show that if $\CC_F(K) \not= \emptyset$
then the minimal level is zero.

By clearing denominators, we may start with an integral $(2,2)$-form.
If this form is $K$-equivalent to an integral form of smaller level,
then our task is to find such a form explicitly.  Define $v(F)$ to be
the minimum of the valuations of the coefficients of $F$. If $v(F)\geq
1$ then we can divide through by $\pi$, reducing the level of $F$.  We
may therefore assume $v(F) = 0$.

Our algorithm for minimising $(2,2)$-forms is described by the
following theorem.

\begin{thm}
\label{thm:min22}
Let $F$ be a non-minimal $(2,2)$-form with $v(F)=0$. Let $f$ be
reduction of $F$ mod $\pi$. Then we are in one of the following three
situations.
\begin{enumerate}
\item The form $f$ factors as a product of binary quadratic forms,
  both of which have a repeated root.  By an $\OK$-equivalence we may
  assume $f=x_2^2 y_2^2$.  Then at least one of the forms
  \begin{align*}
    & \pi^{-2} F(x_1, \pi x_2; y_1, y_2) \\
    & \pi^{-2} F(x_1, x_2; y_1, \pi y_2) \\
    & \pi^{-3} F(x_1, \pi x_2; y_1, \pi y_2)
  \end{align*}
  is an integral $(2,2)$-form of smaller level.
\item The form $f$ factors as a product of binary quadratic forms,
  exactly one of which has a repeated root.  By an $\OK$-equivalence,
  and switching the two sets of variables if necessary, we may assume
  that $f=x_2^2 h(y_1,y_2)$. Then $\pi^{-1} F(x_1, \pi x_2; y_1, y_2)$
  is an integral $(2,2)$-form of the same level.
\item The curve $\CC_f \subset \PP^1 \times \PP^1$ has a unique
  singular point. By an $\OK$-equivalence, this is the point
  $((1:0),(1:0))$. Then $\pi^{-2} F(x_1, \pi x_2; y_1, \pi y_2)$ is an
  integral $(2,2)$-form of the same level.
\end{enumerate}
Moreover the $(2,2)$-form $F$ computed in (ii) or (iii) either has
$v(F) \ge 1$ or has reduction mod $\pi$ of the form specified in (i).
\end{thm}

\begin{Remark}
\label{remFG}
  Let $F$ be an integral $(2,2)$-form, with associated binary quartics
  $G_1$ and $G_2$. It is clear by~\eqref{22->2} that if either $G_1$
  or $G_2$ is minimal then $F$ is minimal.  However the converse is
  not true. For example if $F \equiv (x_1y_1 + x_2y_2)^2
  \pmod{\pi^2}$, then $F$ is minimal by Theorem~\ref{thm:min22}, yet
  we have $G_1 \equiv G_2 \equiv 0 \pmod{\pi^2}$.
\end{Remark}

Exactly as in the case of binary quartics, any non-minimal $(2,2)$-form $F$ is
$\OK$-equivalent to a form whose level can be reduced using diagonal
transformations. Indeed, suppose that $[\lambda,A_1,A_2] \in \GGG(K)$
is a transformation reducing the level.  By clearing denominators, we
may assume that the $A_i$ have entries in $\OK$, not all in $\pi \OK$.
Then writing these matrices in Smith normal form we have
$A_i=Q_iD_iP_i$ where $P_i,Q_i \in \GL_2(\OK)$ and
\[ D_1 = \begin{pmatrix} 1 & 0 \\ 0 & \pi^{a} \end{pmatrix},
\quad 
D_2 = \begin{pmatrix} 1 & 0 \\ 0 & \pi^{b} \end{pmatrix}, \] 
for some integers $a,b \ge 0$. Replacing
$F$ by an $\OK$-equivalent form, it follows that
\[ \pi^{-a-b-1} F(x_1,\pi^{a}x_2;y_1,\pi^{b}y_2) \] 
is an integral $(2,2)$-form.  We say that the pair $(a,b)$ is {\em
  admissible} for $F$.

\begin{lemma}
\label{lem:wts22}
Let $F$ be an integral $(2,2)$-form.  If some pair $(a,b)$ is
admissible for $F$ then at least one
of the following pairs is admissible:
  \[(0,0),\,\, (1,0),\,\, (0,1),\,\, (1,1),\,\, (2,1),\,\, (1,2).\]
\end{lemma}
\begin{proof}
  The coefficients of $F$, arranged as in \eqref{mat:22}, have
  valuations satisfying
  \begin{align*}
    &\geq a+b+1 & &\geq a+1 &&\geq a-b+1 \\
    &\geq b+1 & &\geq 1 &&\geq -b+1 \\
    &\geq -a+b+1 & &\geq -a+1 &&\geq -a-b+1.
  \end{align*}
  Conversely, if the valuations satisfy these inequalities then the
  pair $(a,b)$ is admissible. If $a=b=0$ then we are done as $(0,0)$
  is on the list. If $a \geq 1, b=0$ or $a=0, b \geq 1$, then $(1,0)$
  or $(0,1)$ is admissible.  If $a=b>0$, then $(1,1)$ is
  admissible. If $a>b>0$ or $b>a>0$, then $(2,1)$ or $(1,2)$ is
  admissible.
\end{proof}

\begin{ProofOf}{Theorem~\ref{thm:min22}}
  For the proof we are free to replace the $(2,2)$-form $F$ by an
  $\OK$-equivalent form.  Indeed the transformations specified in the
  statement of the theorem induce well-defined maps on
  $\OK$-equivalence classes, as may be verified using
  \cite[Lemma~4.1]{CFS}.  We may therefore assume that one of the
  pairs $(a,b)$ listed in Lemma~\ref{lem:wts22} is admissible for
  $F$. Since $v(F) = 0$ we cannot have $a=b=0$. By switching the two
  sets of variables, we may assume $a \ge b$. This leaves us with
  three cases.

\medskip

\paragraph{{\bf Case 1}}
We assume $(1,0)$ is admissible for $F$. The coefficients of $F$ have
valuations satisfying
\[
\begin{matrix}
\geq 2 & \geq 2 &\geq 2\\
\geq 1 & \geq 1 &\geq 1\\
\geq 0 & \geq 0 &\geq 0 
\end{matrix}
\]
We have $f = x_2^2 h(y_1,y_2)$ where $h$ is a binary quadratic
form. If $h$ has a repeated root, then the first transformation in (i)
decreases the level.  Otherwise the transformation in (ii) gives a
$(2,2)$-form $F$ with $v(F) \ge 1$.

\medskip

\paragraph{{\bf Case 2}}
We assume $(1,1)$ is admissible for $F$. The coefficients of $F$ have
valuations satisfying
\[
\begin{matrix}
\geq 3 & \geq 2 &\geq 1 \\
\geq 2 & \geq 1 & \geq 0 \\
\geq 1 & \geq 0     &\geq 0 
\end{matrix}
\]
We have
\[
f=x_2y_2(\alpha x_1 y_2 + \beta x_2y_1 + \gamma x_2 y_2)  
\]
for some $\alpha,\beta,\gamma \in k$.  If $\alpha = \beta= 0$ then the
third transformation in (i) decreases the level. If exactly one of the
coefficients $\alpha$ and $\beta$ is zero then the transformation in
(ii) gives a $(2,2)$-form whose reduction mod $\pi$ is either zero, or
of the form specified in (i).  If $\alpha$ and $\beta$ are both
non-zero then $\CC_f \subset \PP^1 \times \PP^1$ has a unique singular
point at $((1:0),(1:0))$. The transformation in (iii) gives a
$(2,2)$-form $F$ with $v(F) \ge 1$.

\medskip

\paragraph{{\bf Case 3}}
We assume $(2,1)$ is admissible for $F$. The coefficients of $F$ have
valuations satisfying
\[
\begin{matrix}
\geq 4 & \geq 3 &  \ge 2 \\
\geq 2 & \geq 1 & = 0 \\
=0 &  \geq 0     &\geq 0 
\end{matrix}
\]
The two valuations indicated are zero, as we would otherwise be in
Case~1 or Case~2.  A calculation shows that $\CC_f \subset \PP^1
\times \PP^1$ has a unique singular point at $((1:0),(1:0))$.  The
transformation in (iii) gives a $(2,2)$-form whose reduction mod $\pi$
is of the form specified in (i).
\end{ProofOf}

The following lemma will be needed in Section~\ref{sec:hypercubes}, in
connection with our study of $2 \times 2 \times 2 \times 2$
hypercubes. 

\begin{lemma}
\label{lemB}
Let $F$ be a non-minimal $(2,2)$-form, and let $f = F$ 
mod $\pi$. 
\begin{enumerate}
\item If $\CC_f \subset \PP^1 \times \PP^1$ is singular at
  $((1:0),(1:0))$, then the coefficients of
  $F$ 
  have valuations satisfying
\[
\begin{matrix}
  \geq 1 & \geq 1 &\geq 1 \\
  \geq 1 & \geq 1 &\geq 0 \\
  \geq 0 & \geq 0 &\geq 0
\end{matrix}
\quad \text{ or } \quad
\begin{matrix}
  \geq 1 & \geq 1 &\geq 0 \\
  \geq 1 & \geq 1 &\geq 0 \\
  \geq 1 & \geq 0 &\geq 0
\end{matrix}
\]
\item If $f = x_2^2 y_2^2$ then the coefficients of $F$ 
  have valuations satisfying
\[
\begin{matrix}
  \geq 2 & \geq 2 &\geq 2 \\
  \geq 1 & \geq 1 &\geq 1 \\
  \geq 1 & \geq 1 &   = 0
\end{matrix}
\quad \text{ or } \quad
\begin{matrix}
  \geq 2 & \geq 1 &\geq 1 \\
  \geq 2 & \geq 1 &\geq 1 \\
  \geq 2 & \geq 1 &   = 0
\end{matrix}
\quad \text{ or } \quad
\begin{matrix}
  \geq 3 & \geq 2 &\geq 1 \\
  \geq 2 & \geq 1 &\geq 1 \\
  \geq 1 & \geq 1 &   = 0
\end{matrix}
\]
\end{enumerate}
\end{lemma}
\begin{proof}
(i) The singular point forces $a_{11} \equiv a_{12} \equiv a_{21} \equiv 0
\pmod{\pi}$. The vanishing of the invariants $u$ and $v$ in 
Lemma~\ref{lem:inv22} gives
\[ 8 a_{13} a_{31} + a_{22}^2 \equiv a_{13}a_{22}a_{31} \equiv 0 \pmod{\pi}.\]
It follows that $a_{22} \equiv 0 \pmod{\pi}$. The same lemma shows that
$(\xi,\eta) = (a_{13} a_{31},0)$ is a singular point on the curve with 
Weierstrass equation $y^2 \equiv x^2(x - a_{13} a_{31}) \pmod{\pi}$.
Therefore $a_{13} a_{31} \equiv 0 \pmod{\pi}$.  

(ii) The proof of Theorem~\ref{thm:min22} shows that $F$ is
$\OK$-equivalent to a $(2,2)$-form $F_1$ with
\begin{equation}
\label{eqn:Ftilde}
F_1 (x_1, \pi^a x_2 ; y_1 \pi^b y_2) \equiv 0 \pmod{\pi^{a+b+1}} 
\end{equation}
for some $(a,b) = (1,0)$, $(0,1)$ or $(1,1)$. Working mod $\pi$ we
have $F_1 \equiv x_2^2 h(y_1,y_2)$, $g(x_1,x_2) y_2^2$ or
$x_2y_2(\alpha x_1 y_2 + \beta x_2y_1 + \gamma x_2 y_2)$.  In the last
case it follows from our assumption $F \equiv x_2^2 y_2^2 \pmod{\pi}$
that $\alpha = \beta = 0$. The equivalence relating $F$ and
$F_1$ must now fix the points $(x_1:x_2) =(1:0)$ mod $\pi$,
$(y_1:y_2) =(1:0)$ mod $\pi$, or both.  It follows that $F$ also
satisfies~\eqref{eqn:Ftilde}.
\end{proof}


\section{$3\times3\times3$ Rubik's cubes}
\label{sec:cubes}

We consider polynomials in $x_1,x_2,x_3, y_1, y_2, y_3,z_1, z_2, z_3$
that are linear in each of the three sets of variables. Such a form
may be represented as
\begin{equation*}
  \sum_{1\leq i,j,k \leq 3} s_{ijk} x_i y_j z_k
\end{equation*}
where $S = (s_{ijk})$ is a $3 \times 3 \times 3$ cubical matrix.
A {\em Rubik's cube} $S$ may be partitioned into three $3\times3$
matrices in three distinct ways:
\begin{enumerate}
\item $M^1 = (s_{1jk})$ is the front face, $N^1=(s_{2jk})$ is the
  middle slice and $P^1=(s_{3jk})$ is the back face.
\item $M^2 = (s_{i1k})$ is the top face, $N^2=(s_{i2k})$ is the middle
  slice and $P^2=(s_{i3k})$ is the bottom face.
\item $M^3 = (s_{ij1})$ is the left face, $N^3=(s_{ij2})$ is the
  middle slice and $P^3=(s_{ij3})$ is the right face.
        \end{enumerate}
        To each slicing $(M^i,N^i,P^i)$, we may associate a ternary
        cubic form
\begin{equation*}
  F_{i}(x,y,z)=\det(M^{i}x+N^{i}y+P^{i}z).
\end{equation*}
Following~\cite[Section 2]{CFS} we scale the invariants $c_4, c_6,
\Delta$ of a ternary cubic so that $c_4(xyz) =1$, $c_6(xyz)=-1$ and
$c_4^3 - c_6^2 = 1728\Delta$. It may be checked that the $F_i$ have
the same invariants. We define $c_4(S) = c_4(F_i)$, $c_6(S) =
c_6(F_i)$ and $\Delta(S) = \Delta(F_i)$.

If $S$ is defined over a field and $\Delta(S) \not = 0$ then each of
the $F_i$ defines a smooth curve of genus $1$ in $\PP^2$.  These
curves are isomorphic, although not in a canonical way.  (See
\cite[Section 3.2]{BH} for further details.) We write $\CC_S$ to
denote any one of them.

Let $R$ be a ring. For each $1 \le i \le 3$ there is an action of
$\GL_3(R)$ on the space of Rubik's cubes over $R$ given by
\begin{align*}
A  = (a_{ij}) : (M^i,N^i,P^i) \mapsto (a_{11}& M^i+a_{12}N^{i}+a_{13}P^{i}, \\ 
& a_{21}M^i+a_{22}N^{i}+a_{23}P^{i},a_{31}M^i+a_{32}N^{i}+a_{33}P^{i}). 
\end{align*}
These actions commute, and so give an action of $\GGG(R) =
\GL_3(R)^3$.  We say that $3 \times 3 \times 3$ cubes are {\em
  $R$-equivalent} if they belong to the same orbit for this action.
If $[A_1,A_2,A_3] \cdot S = S'$ then the associated ternary cubics are
related by
\begin{equation}
\label{333->3}
     F'_i(x,y,z) = \det(A_jA_k) F_i ((x,y,z)A_i) 
\end{equation}
where $\{i,j,k\} = \{1,2,3\}$.

A polynomial $I \in \Z[s_{ijk}]$ is an {\em invariant} of weight $p$
if
\[ I([A_1,A_2,A_3] \cdot S)=(\det A_1 \det A_2 \det A_3)^{p} I(S) \]
for all $[A_1,A_2,A_3] \in \GGG(\C)$.  In particular the polynomials
$c_4, c_6$ and $\Delta$ 
are invariants of weights $4$, $6$ and $12$. Over a field of
characteristic not $2$ or $3$, the invariants determine a pair $(E,P)$
where $E$ is an elliptic curve (the Jacobian of $\CC_S$) and $P$ is a
marked point on $E$. The next lemma gives formulae for $E$ and $P$
that work in all characteristics.

\begin{lemma}
\label{lem:inv333}
There exist $\xi,\eta,a_1,a_2,a_3,a_4,a_6 \in \Z[s_{ijk}]$ such that
\begin{enumerate}
\item We have $c_4 = b_2^2 - 24 b_4$ and $c_6 = -b_2^3 + 36 b_2 b_4 -
  216 b_6$, where $b_2 = a_1^2 + 4 a_2$, $b_4 = a_1 a_3 + 2 a_4$ and
  $b_6 = a_3^2 + 4 a_6$,
\item The polynomials $u = 12 \xi + a_1^2 + 4 a_2$ and $v = 2 \eta +
  a_1 \xi + a_3$ are invariants of weights $2$ and $3$ satisfying
  $(108 v)^2 = (3 u)^3 - 27 c_4 (3 u) - 54 c_6$.
\item We have $\eta^2 + a_1 \xi \eta + a_3 \eta = \xi^3 + a_2 \xi^2 +
  a_4 \xi + a_6$.
\end{enumerate}
\end{lemma}

\begin{proof}
  We define matrices $A, B, C$ by the rule
  \[ (\adj(\lambda N^1 + \mu P^1)) M^1 = \lambda^2 A + \lambda \mu B +
  \mu^2 C. \]
  We put $\xi = -\tr(A C)$ and $\eta = - \tr(C B A)$. \\
  (i) We put
  \begin{align*}
    a_1 &= \tr(B), \\
    a_2 &= \tr(A C) + \tr(A) \tr(C) - \tr(\adj(B)), \\
    a_3 &= \tr(A B C) + \tr(C B A) + \tr(A C) \tr(B).
  \end{align*}
  Since we already defined $c_4$ and $c_6$, we could now in principle
  solve for $a_4$ and $a_6$.  However it is simpler to argue as
  follows. Let $a'_1, \ldots, a'_6$ be the $a$-invariants (as defined
  in \cite[Lemma 2.9]{CFS}) of the ternary cubic $F_1$. We checked by
  computer algebra that there exist $r,s,t \in \Z[s_{ijk}]$ satisfying
\begin{align*}
  a'_1 &= a_1 + 2s, \\
  a'_2 &= a_2 - s a_1 + 3r - s^2, \\
  a'_3 &= a_3 + r a_1 + 2t.
\end{align*}
It follows by the transformation formulae for Weierstrass equations
(see \cite{Sil}) 
that $a_4,a_6 \in \Z[s_{ijk}]$.  Note that our reason
for working with $a_1, \ldots ,a_6$, in preference to $a'_1, \ldots
,a'_6$, is that this helped us find 
particularly simple expressions for $\xi$ and $\eta$. \\
(ii) The invariants $u$ and $v$ were denoted $4 c_6$ and $c_9$ in
\cite[Section 5.1.3]{BH}. In fact we have
$v = \tr(A B C) - \tr(C B A)$. \\
(iii) This follows from (i) and (ii) exactly as in \cite[Chapter
III]{Sil}.
\end{proof}

A Rubik's cube $S$ is {\em integral} if it has coefficients in $\OK$,
and {\em non-singular} if $\Delta(S) \not= 0$.
\begin{lemma}
\label{lem:lev333}
Let $S$ be a non-singular integral Rubik's cube. Let $(E,P)$ be the
pair specified in Lemma~\ref{lem:inv333}. Then
\[ v(\Delta(S)) = v(\Delta_E) + 12 \kappa(P) + 12 \ell(S) \] where
$\ell(S) \ge 0$ is an integer we call the {\em level}.
\end{lemma}
\begin{proof}
  The proof is identical to that of Lemma~\ref{lem:lev22}.
\end{proof}

In this section we give an algorithm for minimising Rubik's cubes. In
Section~\ref{sec:minthm} we show that if $\CC_S(K) \not= \emptyset$
then the minimal level is zero.

We say that an integral cube $S$ is {\em saturated} if for each
$i=1,2,3$ the matrices $M^i,N^i,P^i \in \Mat_3(\OK)$ are linearly
independent mod $\pi$.  If an integral cube is not saturated, then it
is obvious how we may decrease the level.

Our algorithm for minimising $3 \times 3 \times 3$ cubes is described
by the following theorem.

\begin{thm}
\label{thm:min333}
Let $S$ be a non-minimal saturated Rubik's cube. Let $F_1, F_2, F_3$
be the associated ternary cubics, and $f_1,f_2,f_3$ their reductions
mod $\pi$.  Then we are in one of the following two situations.
\begin{enumerate}
\item Two or more of the $f_i$ are non-zero and have a repeated linear factor, say
  $f_1$ and $f_2$ are divisible by $z^2$. We apply a transformation
  \[ \left[\begin{pmatrix} 1 & & \\
      & 1 & \\ & & \pi \end{pmatrix},
    \begin{pmatrix} 1 & & \\ & 1 & \\ & & \pi
    \end{pmatrix}, A_3 \right] \] where $A_3 \in \GL_3(K)$ is chosen
  such that $M^3,N^3,P^3 \in \Mat_3(\OK)$ are linearly independent mod
  $\pi$.
\item Two or more of the $f_i$ define a curve with a unique singular
  point, say $f_1$ and $f_2$ define curves with singular points at
  $(1:0:0)$. We apply a transformation
  \[ \left[ \begin{pmatrix} 1 & & \\ & \pi & \\ & & \pi
\end{pmatrix},  
\begin{pmatrix} 1 & & \\ & \pi & \\ & & \pi
\end{pmatrix}, A_3 \right] \] where $A_3 \in \GL_3(K)$ is chosen such
that $M^3,N^3,P^3 \in \Mat_3(\OK)$ are linearly independent mod $\pi$.
\end{enumerate}
The procedures in (i) and (ii) give an integral cube of the same or
smaller level. Repeating these procedures either gives a non-saturated
cube or decreases the level after at most three iterations.
\end{thm}

\begin{Remark}
  Let $S$ be an integral Rubik's cube, with associated ternary cubics
  $F_1, F_2, F_3$. It is clear by~\eqref{333->3} that if any of the
  $F_i$ are minimal then $S$ is minimal.  However the converse is not
  true. For example if $S \equiv (\varepsilon_{ijk}) \pmod{\pi}$,
  where $\varepsilon_{ijk}$ is the Levi-Civita symbol (as appears in
  the definition of the cross product), then $S$ is minimal by
  Theorem~\ref{thm:min333}, yet we have $F_1 \equiv F_2 \equiv F_3
  \equiv 0 \pmod{\pi}$.
\end{Remark}

Exactly as in the case of $(2,2)$-forms, any non-minimal Rubik's cube $S$ is
$\OK$-equivalent to a cube whose level can be reduced using diagonal
transformations. Indeed, suppose that $[\pi^{-s} A_1,A_2,A_3] \in
\GGG(K)$ is a transformation reducing the level.  By clearing
denominators, we may assume that the $A_i$ have entries in $\OK$, not
all in $\pi \OK$.  Then writing these matrices in Smith normal form we
have $A_i=Q_iD_iP_i$ where $P_i,Q_i \in \GL_3(\OK)$ and
\[ D_i=\left( \begin{matrix} \pi^{a_{1i}} & 0 & 0 \\ 0 & \pi^{a_{2i}}
    & 0 \\ 0 & 0 & \pi^{a_{3i}} \end{matrix}\right) \] with
$\min(a_{1i},a_{2i},a_{3i})=0$. If this transformation reduces the
level then $\sum a_{ij} < 3s$. In fact, by increasing one of the
$a_{ij}$, we may assume $\sum a_{ij} = 3s-1$.  We will from now on
assume $a_{11} = a_{12} = a_{13} =0$.  If the new cube has
coefficients in $\OK$ then we say that the tuple
$(a_{21},a_{31};a_{22},a_{32};a_{23},a_{33})$ is {\em admissible} for
$S$.

\begin{lemma}
\label{findtypes}
Let $S$ be a non-minimal Rubik's cube. Then after permuting the three slicings, and 
replacing $S$ by an $\OK$-equivalent cube, 
at least one of the following tuples is admissible.
\begin{align*}
  \tau_1 &= ( 1, 1; 0, 0; 0, 0 ), & 
  \tau_2 &= ( 0, 1; 0, 1; 0, 0 ), &
  \tau_3 &= ( 1, 2; 0, 1; 0, 1 ), \\
\tau_4 &= ( 1, 1; 1, 1; 0, 1 ), &
 \tau_5 &= ( 1, 2; 1, 2; 1, 1 ), &
 \tau_6 &= ( 2, 3; 1, 2; 1, 2 ).
\end{align*}
\end{lemma}

\begin{proof}
  We define the set of {\em weights}
  \[ \W = \left\{ (A,s) \in \Mat_3(\Z) \times \Z \,\, \left|
      \begin{array}{c} a_{11}= a_{12} = a_{13} = 0, \\
        a_{ij} \ge 0 \text{ for all } i,j, \\ \sum a_{ij} = 3s
        -1 \end{array} \right. \right\}. \] If $(A,s) \in \W$ then
  $(a_{21},a_{31};a_{22},a_{32};a_{23},a_{33})$ is admissible for $S$
  if and only if
  \[ v(s_{ijk}) \ge \max(s-a_{i1} - a_{j2}-a_{k3},0) \] for all $i,j,k
  \in \{1,2,3\}$. We define a partial order on $\W$ by $(A,s) \le
  (A',s')$ if
\[
\max(s-a_{i1} - a_{j2}-a_{k3},0) \le \max(s'-a'_{i1} -
a'_{j2}-a'_{k3},0)
\]
for all $i,j,k \in \{1,2,3\}$. A computer calculation, using
Lemma~\ref{lem:getbound} below, shows that $(\W,\le)$ has exactly $81$
minimal elements.  By an $\OK$-equivalence we may assume $a_{2i} \le
a_{3i}$ for $i=1,2,3$, and by permuting the three slicings of $S$ we
may assume $a_{31} \geq a_{32} \geq a_{33}$.  Only $8$ of the $81$
minimal elements satisfy these additional conditions. These are the
$6$ elements listed in the statement of the lemma, together with two
more that are the same as $\tau_4$ up to permuting the slicings.
\end{proof}

\begin{lemma} 
\label{lem:getbound}
If $(A,s) \in \W$ is minimal then $s \le 10$.
\end{lemma}

\begin{proof}
  Suppose for a contradiction that $(A,s)$ is minimal with $s >
  10$. Without loss of generality we have
  \begin{equation}
  \label{eqn:1}
  a_{21} \le a_{31}, \,\,\, a_{22} \le a_{32}, \,\,\, a_{23} \le a_{33} 
  \,\, \text{ and } \,\, a_{31} \geq a_{32} \geq a_{33}. 
\end{equation}
Since $6 a_{31} \ge \sum a_{ij} = 3s - 1$ we certainly have $a_{31} >
3$.  Let $A'$ be the matrix obtained from $A$ by replacing $a_{31}$ by
$a_{31}-3$. Then $(A',s-1) \in \W$, and by our minimality assumption
$(A',s-1) \not\le (A,s)$. Therefore
\[
   \max(s-1-a'_{i1} - a'_{j2}-a'_{k3},0) > 
    \max(s-a_{i1} - a_{j2}-a_{k3},0)
\]
for some $i,j,k \in \{1,2,3\}$. Since we only changed the entry
$a_{31}$ we must have $i=3$ and $s-1-(a_{31}-3) > 0$.  Therefore
\begin{equation}
\label{eqn:2}
s+1 \ge a_{31}.
\end{equation}
The following inequalities are obtained in an entirely analogous way:
\begin{enumerate}
\item If $a_{33}>0$ then by considering
  $(a_{21},a_{31}-1;a_{22},a_{32}-1;a_{23},a_{33}-1)$, we have $s \ge
  a_{32}+a_{33}$.
\item If $a_{21},a_{22},a_{23} >0$ then by considering
  $(a_{21}-1,a_{31}-1,a_{22}-1,a_{32}-1,a_{23}-1,a_{33}-1)$, we have
  $s \ge a_{21}+a_{22}+a_{23}$.
\item If $a_{22}>0$ then by considering
  $(a_{21},a_{31}-1;a_{22}-1,a_{32}-1;a_{23},a_{33})$, we have $s \ge
  a_{31}+a_{22}$.
\item If $a_{21},a_{32}>0$ then by considering
  $(a_{21}-1,a_{31}-1;a_{22},a_{32}-1;a_{23},a_{33})$, we have $s \ge
  a_{21}+a_{32}$.
\item If $a_{23}>0$ then by considering
  $(a_{21},a_{31}-1;a_{22},a_{32};a_{23}-1,a_{33}-1)$, we have $s \ge
  a_{31}+a_{23}$.
\end{enumerate}

We now claim that if $a_{33} >0$ then $s \ge a_{21}+a_{22}+a_{23}$.
Indeed if $a_{21},a_{22},a_{23} >0$ then this is (ii). If $a_{21}=0$
then we instead use (i). If $a_{21} >0$ and $a_{23}=0$ then (noting that
$a_{32} \ge a_{33} > 0$) we instead use (iv).
If $a_{23}>0$ and $a_{22} =0$ then we instead use (v).

To complete the proof of the lemma, we first suppose $a_{33} >0$.
Then the inequalities in (i) and (ii) hold without further hypothesis.
We weaken the inequalities (iii), (iv) and (v) to
\begin{align}
\label{weak1}
s+1 &\ge a_{31}+a_{22} \\
\label{weak2}
s+1 &\ge a_{21}+a_{32} \\
s+1 &\ge a_{31}+a_{23}
\end{align}
so that in cases where some of the $a_{ij}$ are zero, these still hold
by~\eqref{eqn:1} and \eqref{eqn:2}. Adding together all five
inequalities gives
\[ 5 s + 3 + a_{33} \ge 2 \sum a_{ij} = 2(3s-1) \] and hence $a_{33}
\ge s-5$. Using (i) again gives
\[ s \ge a_{32}+a_{33} \ge 2 a_{33} \ge 2(s-5) \] and hence $s \le
10$, as required.

If $a_{33} =0$ then we still have \eqref{weak1} and \eqref{weak2}
giving $2(s+1) \ge \sum a_{ij} = 3 s - 1$, and hence $s \le 3$.
\end{proof}

\begin{ProofOf}{Theorem~\ref{thm:min333}}
  We represent $S$ as a triple of matrices $A,B,C$, say.
  \[
\begin{matrix}
  A_{11} &  A_{12} & A_{13}\\
  A_{21} &  A_{22} & A_{23}\\
  A_{31} &  A_{32} & A_{33}\\
\end{matrix} \hspace{2em}
\begin{matrix}
 B_{11} &  B_{12} & B_{13}\\
 B_{21} &  B_{22} & B_{23}\\
 B_{31} &  B_{32} & B_{33}\\
\end{matrix} \hspace{2em}
\begin{matrix}
 C_{11} &  C_{12} & C_{13}\\
 C_{21} &  C_{22} & C_{23}\\
 C_{31} &  C_{32} & C_{33}\\
\end{matrix} 
\]
The action of $\GGG(K) =
\GL_3(K)^3$ 
may be described as follows.  The first factor replaces $A$, $B$, $C$
by linear combinations of these matrices. The second factor acts by
row operations (applied to $A$, $B$, $C$ simultaneously), and the
third factor acts by column operations.

We may assume one of the tuples $\tau_1, \ldots,\tau_6$ in
Lemma~\ref{findtypes} is admissible for $S$.  We therefore split into
these $6$ cases.

\medskip

\paragraph{{\bf Case 1}}
We assume $(1,1;0,0;0,0)$ is admissible for $S$. Then the entries of
$A$ have valuation at least one, and so the cube $S$ is not saturated.

\medskip

\paragraph{{\bf Case 2}}
We assume $(0,1;0,1;0,0)$ is admissible for $S$. The entries of $A,B$
and $C$ have valuations satisfying
    \begin{equation*}
    \begin{matrix}
 \geq 1 &  \geq 1 & \geq 1\\
\geq 1 &  \geq 1 & \geq 1\\
\geq 0 &  \geq 0 & \geq 0
\end{matrix} \hspace{2em}
\begin{matrix}
 \geq 1 &  \geq 1 & \geq 1\\
\geq 1 &  \geq 1 & \geq 1\\
\geq 0 &  \geq 0 & \geq 0
\end{matrix} \hspace{2em}
\begin{matrix}
\geq 0 &  \geq 0 & \geq 0\\
\geq 0 &  \geq 0 & \geq 0\\
\geq 0 &  \geq 0 & \geq 0\\
\end{matrix} 
\end{equation*}
Since $S$ is saturated we may assume by column operations that
$v(C_{11})=0$, $v(C_{12}) \ge 1$ and $v(C_{13}) \ge 1$. Subtracting a
multiple of the first row from the second row gives $v(C_{21}) \ge 1$,
and again by column operations $v(C_{22})=0$ and $v(C_{23}) \ge 1$.
Subtracting multiples of the first two rows from the third, the
valuations now satisfy
   \begin{equation*}
    \begin{matrix}
      \geq 1 &  \geq 1 & \geq 1\\
      \geq 1 &  \geq 1 & \geq 1\\
      \geq 0 &  \geq 0 & \geq 0\\
      \end{matrix} \hspace{2em}
    \begin{matrix}
      \geq 1 &  \geq 1 & \geq 1\\
      \geq 1 &  \geq 1 & \geq 1\\
      \geq 0 &  \geq 0 & \geq 0\\
      \end{matrix} \hspace{2em}
    \begin{matrix}
         = 0 &  \geq 1 & \geq 1\\
      \geq 1 &  = 0    & \geq 1\\
      \geq 1 &  \geq 1 & \geq 0\\
    \end{matrix}
\end{equation*}
We compute $f_{1} = C_{11}C_{22} z^2 (A_{33} x + B_{33} y + C_{33} z)
\mod{\pi}$.  Since $S$ is saturated it follows that $f_1$ is
nonzero. The same argument shows that $f_2$ has a repeated factor and
is nonzero. On the other hand we have $f_3=0$.  The procedure in (i)
multiplies $C$ and the third row by $\pi$, and then divides the cube
by $\pi$.  This transformation decreases the level.

\medskip

\paragraph{{\bf Case 3}}
We assume $(1,2;0,1;0,1)$ is admissible for $S$. The entries of $A,B$
and $C$ have valuations satisfying
 \begin{equation*}
    \begin{matrix}
      \geq 2 &  \geq 2 & \geq 1\\
      \geq 2 &  \geq 2 & \geq 1\\
      \geq 1 & \geq 1 & \geq 0
\end{matrix} \hspace{2em}
\begin{matrix}
  \geq 1 &  \geq 1 & \geq 0\\
  \geq 1 &  \geq  1& \geq 0\\
  \geq 0 & \geq 0 & \geq 0
\end{matrix} \hspace{2em}
\begin{matrix}
  \geq 0 &  \geq 0 & \geq 0\\
  \geq 0 &  \geq 0 & \geq 0\\
  \geq 0 & \geq 0 & \geq 0
\end{matrix} 
\end{equation*}
Since $S$ is saturated we have $v(A_{33})=0$. If $B_{13} \equiv B_{23}
\equiv 0 \pmod{\pi}$ then we are in Case 2, and likewise if $B_{31}
\equiv B_{32} \equiv 0 \pmod{\pi}$. By operating on the first two rows
and columns, and then subtracting a multiple of $A$ from $B$, the
valuations now satisfy
 \begin{equation*}
    \begin{matrix}
      \geq 2 &  \geq 2 & \geq 1\\
      \geq 2 &  \geq 2 & \geq 1\\
      \geq 1 & \geq 1 & = 0
\end{matrix} \hspace{2em}
\begin{matrix}
  \geq 1 &  \geq 1 & \geq 1\\
  \geq 1 &  \geq  1& = 0\\
  \geq 1 & = 0 & \geq 1
\end{matrix} \hspace{2em}
\begin{matrix}
  \geq 0 &  \geq 0 & \geq 0\\
  \geq 0 &  \geq 0 & \geq 0\\
  \geq 0 & \geq 0 & \geq 0
\end{matrix} 
\end{equation*}
Working mod $\pi$ we compute
\begin{align*}
  f_1 &= -B_{23} B_{32} C_{11} y^2 z + z^2( \,\, \cdots ) \\
  f_2 &= -A_{33} B_{32} z^2 ( C_{11} x + C_{21} y + C_{31}z) \\
  f_3 &= -A_{33} B_{23} z^2 ( C_{11} x + C_{12} y + C_{13}z)
\end{align*}
Since $S$ is saturated, it is clear that $f_2$ and $f_3$ are nonzero.

We note that multiplying $C$, the last row and the last column by
$\pi$, and then dividing the whole cube by $\pi$, gives an integral
model of the same level which is not saturated. These transformations
are carried out by the procedure in (i), except possibly in the case
where $f_1$ has a repeated factor, and this factor is not $z^2$. In
this remaining case $v(C_{11}) = 0$. We may assume by row and column
operations that $C_{12} \equiv C_{13} \equiv C_{21} \equiv C_{31}
\equiv 0 \pmod{\pi}$.  Subtracting multiples of $A$ and $B$ from $C$
gives $C_{32} = C_{33} = 0 \pmod{\pi}$.  Now $f_1 = C_{11} z ( A_{33}
C_{22} xz - B_{23} B_{32} y^2 - B_{32} C_{23} y z)$, and so $C_{22}
\equiv C_{23} \equiv 0 \pmod{\pi}$.

If the procedure in (i) picks $f_1$ and $f_2$ then we multiply $B$ and
the last row by $\pi$.  Dividing the last two columns by $\pi$ gives a
model of the same level with valuations satisfying
 \begin{equation*}
    \begin{matrix}
      \geq 2 &  \geq 1 & \geq 0\\
      \geq 2 &  \geq 1 & \geq 0\\
      \geq 2 & \geq 1 & = 0
\end{matrix} \hspace{2em}
\begin{matrix}
  \geq 2 &  \geq 1 & \geq 1\\
  \geq 2 &  \geq  1& = 0\\
  \geq 3 & = 1 & \geq 2
\end{matrix} \hspace{2em}
\begin{matrix}
  = 0 &  \geq 0 & \geq 0\\
  \geq 1 &  \geq 0 & \geq 0\\
  \geq 2 & \geq 1 & \geq 1
\end{matrix} 
\end{equation*}
Since the first two columns of $A$ and $B$ are divisible by $\pi$, we
are now in Case~2. The case where the procedure in (i) picks $f_1$ and
$f_3$ works in the same way.

\medskip

\paragraph{{\bf Case 4}}
We assume $(1,1;1,1;0,1)$ is admissible for $S$. The entries of $A,B$
and $C$ have valuations satisfying
\begin{equation*}
    \begin{matrix}
 \geq 2 &  \geq 2 & \geq 1\\
\geq 1 &  \geq 1 & \geq 0\\
\geq 1 &  \geq 1 & \geq 0
\end{matrix} \hspace{2em}
\begin{matrix} 
 \geq 1 &    \geq 1 & \geq 0\\
 \geq 0 &  \geq  0 & \geq 0\\
\geq 0 &  \geq 0  & \geq 0
\end{matrix} \hspace{2em}
\begin{matrix}
\geq 1 &  \geq 1 & \geq 0\\
\geq 0 &  \geq 0 & \geq 0\\
 \geq 0 &  \geq 0 & \geq 0
\end{matrix} 
\end{equation*}
Working mod $\pi$ we compute
\[ f_1 = (B_{13} y + C_{13} z) \left| \begin{pmatrix} B_{21} & B_{22} \\
    B_{31} & B_{32} \end{pmatrix} y + \begin{pmatrix} C_{21} & C_{22} \\
    C_{31} & C_{32} \end{pmatrix} z \right|, \] and
\[ f_2 = (A_{23} y + A_{33} z) \left| \begin{pmatrix} B_{21} & B_{22} \\
    C_{21} & C_{22} \end{pmatrix} y + \begin{pmatrix} B_{31} & B_{32} \\
    C_{31} & C_{32} \end{pmatrix} z \right|. \] Since $S$ is
saturated, the linear factors $\ell_1 = B_{13} y + C_{13} z$ and
$\ell_2 = A_{23} y + A_{33} z$ cannot be identically zero. Let $q_1$
and $q_2$ be the quadratic factors. These are binary quadratic forms
associated to the same $2 \times 2 \times 2$ cube. In particular $q_1$
and $q_2$ have the same discriminant, say $\delta$.  If this cube is
not saturated, it is easy to see we are in Case 1 or Case 2. Therefore
$f_1$ and $f_2$ are nonzero.

Replacing $B$ and $C$ by suitable linear combinations, and likewise
the last two rows, we may suppose that the linear factors $\ell_1$ and
$\ell_2$ are multiples of $z$, i.e.
\begin{equation}
\label{makeitz}
B_{13} \equiv A_{23} \equiv 0  \pmod{\pi}
\end{equation}
Under this assumption $f_3 = -A_{33} C_{13} z^2 ( B_{21} x + B_{22}y +
B_{23} z)$, and this is nonzero as we would otherwise be in Case 2.

If $f_1$ and $f_2$ don't have repeated factors, then each defines a
curve with a unique singular point at $(1:0:0)$. The procedure in (ii)
multiplies $B$, $C$ and the last two rows by $\pi$.  The level is then
reduced using columns operations, in exactly the way suggested by the
definition of Case 4.

Now suppose that at least one of the forms $f_1$ and $f_2$ has a repeated
factor. Then the procedure in (i) is applied. We say we are in the
{\em good situation} if the two of the $f_i$ chosen are multiples of
$z^2$ and $B_{21} \equiv B_{22} \equiv 0 \pmod{\pi}$. Indeed in the
good situation, the procedure in (i) reduces us to Case 1 or Case 2.

Suppose that $f_1$ and $f_3$ are chosen. Dropping the
assumption~\eqref{makeitz} we may assume that $f_1$ has repeated
factor $z^2$. Then $q_1$ has no $y^2$ term and by row operations we
reach the good situation. The case where $f_2$ and $f_3$ are chosen is
similar. Finally we suppose that $f_1$ and $f_2$ are chosen.  If $q_1$
has a factor $z$, we may assume as above that $B_{21} \equiv B_{22}
\equiv 0 \pmod{\pi}$.  But then $q_2$ has a factor $z$. So if
$\delta=0$, i.e.  $q_1$ and $q_2$ each have a repeated factor, then we
reach the good situation.  Otherwise we make the
assumption~\eqref{makeitz}, and deduce that $f_1$ and $f_2$ are now
multiples of $z^2$.  The procedure in (i) multiplies $C$ and the last
row by $\pi$.  The only coefficients not to vanish mod $\pi$ are now
those in the second row of $B$. It follows that after suitable column
operations the level is preserved and we are reduced to Case 2 or Case 3.

\medskip

\paragraph{{\bf Case 5}}
We assume $(1,2;1,2;1,1)$ is admissible for $S$. The entries of 
$A,B$ and $C$ have valuations satisfying
 \begin{equation*}
    \begin{matrix}
 \geq 3 &  \geq 2 & \geq 2\\
\geq 2 & \geq 1 &  \geq 1 \\
\geq 1 & \geq 0 &  \geq 0 
\end{matrix} \hspace{2em}
\begin{matrix}
 \geq 2 & \geq 1 & \geq 1 \\
\geq 1 & \geq 0 &  \geq 0  \\
 \geq 0 &  \geq  0 & \geq 0
\end{matrix} \hspace{2em}
\begin{matrix}
\geq 1 & \geq 0 &  \geq 0 \\
\geq 0 & \geq 0 &  \geq 0 \\
\geq 0 &  \geq 0 & \geq 0
\end{matrix} 
\end{equation*}
Since $S$ is saturated, we may assume by column operations that
$v(A_{32}) \ge 1$ and $v(A_{33})=0$.  Then $v(B_{31}) = v(C_{12}) =
v(C_{21})=0$, otherwise we would be in Case~4. By row and column
operations, and subtracting multiples of $A$ from $B$ and $C$ we
reduce to the case
 \begin{equation*}
    \begin{matrix}
      \geq 3 &  \geq 2 & \geq 2\\
      \geq 2 & \geq 1 &  \geq 1 \\
      \geq 1 & \geq 1 & = 0
\end{matrix} \hspace{2em}
\begin{matrix}
  \geq 2 & \geq 1 & \geq 1 \\
  \geq 1 & \geq 0 &  \geq 0  \\
  = 0 & \geq 1 & \geq 1
\end{matrix} \hspace{2em}
\begin{matrix}
  \geq 1 & = 0 &  \geq 1 \\
  = 0 & \geq 1 &  \geq 1 \\
  \geq 1 & \geq 1 & \geq 1
\end{matrix} 
\end{equation*}
Working mod $\pi$ we compute
\begin{align*}
  f_1 &= C_{12} z ( B_{31} B_{23} y^2 - A_{33} C_{21} x z) \\
  f_2 &= -A_{33} z ( B_{22} C_{21} y^2  - B_{31} C_{12} xz) \\
  f_3 &= -A_{33} C_{12} y z (B_{22}y + B_{23} z)
\end{align*}

If $B_{22} \not\equiv 0 \pmod{\pi}$ and $B_{23} \not\equiv 0
\pmod{\pi}$ then $f_1,f_2,f_3$ each define a curve with a unique
singular point at $(1:0:0)$. If we multiply $B$, $C$, the last two
rows and the last two columns by $\pi$, then the cube is divisible by
$\pi^2$. From this we see that whichever two of the $f_i$ are chosen
by the procedure in (ii), the level is preserved and we are reduced to
Case~2.

If $B_{22} \not\equiv 0 \pmod{\pi}$ and $B_{23} \equiv 0 \pmod{\pi}$
then $f_1$ and $f_3$ have repeated factors but $f_2$ does not. The
procedure in (i) multiplies $C$ and the middle column by $\pi$. Then
dividing the first two rows by $\pi$ preserves the level and reduces
us to Case~4 with $\delta=0$.  The observation that $\delta=0$ is
needed to show that at most three iterations are required, as claimed
in the statement of the theorem.

If $B_{22} \equiv 0 \pmod{\pi}$ and $B_{23} \not\equiv 0 \pmod{\pi}$
then we switch the first two slicings (i.e. $A,B,C$ are replaced by
the matrices formed from the first, second, third rows). Then
switching the last two columns brings us to the situation considered
in the previous paragraph.

Finally, if $B_{22} \equiv B_{23} \equiv 0 \pmod{\pi}$ then we are
already in Case 2.

\medskip

\paragraph{{\bf Case 6}}
We assume $(2,3;1,2;1,2)$ is admissible for $S$. The entries of $A,B$
and $C$ have valuations satisfying
 \begin{equation*}
    \begin{matrix}
      \geq 4 &  \geq 3 & \geq 2\\
      \geq 3 & \geq 2 &  \geq 1 \\
      \geq 2 & \geq 1 & \geq 0
\end{matrix} \hspace{2em}
\begin{matrix}
  \geq 2 & \geq 1 &   \geq 0 \\
  \geq 1 & \geq 0 &  \geq 0  \\
  \geq0 & \geq 0 & \geq 0
\end{matrix} \hspace{2em}
\begin{matrix}
  \geq 1 & \geq0 &  \geq 0 \\
  \geq 0 & \geq 0 &  \geq 0 \\
  \geq 0 & \geq 0 & \geq 0
\end{matrix} 
\end{equation*}
Since $S$ is saturated, we have $v(A_{33})=0$. Then $v(B_{22})=0$,
otherwise we would be in Case~3. We also have $v(C_{12}) =
v(C_{21})=0$, otherwise we would be in Case~4, and $v(B_{13}) =
v(B_{31}) =0$ otherwise we would be in Case~5. By row and column
operations, and subtracting multiples of $A$ from $B$ and $C$ we
reduce to the case
 \begin{equation*}
    \begin{matrix}
      \geq 4 &  \geq 3 & \geq 2\\
      \geq 3 & \geq 2 &  \geq 1 \\
      \geq 2 & \geq 1 & =0
\end{matrix} \hspace{2em}
\begin{matrix}
  \geq 2 & \geq 1 &   = 0 \\
  \geq 1 & =0 &  \geq 0  \\
  = 0 & \geq 0 & \geq 1
\end{matrix} \hspace{2em}
\begin{matrix}
  \geq 1 & =0 &  \geq 1 \\
  =0 & \geq 1 &  \geq 1 \\
  \geq 1 & \geq 1 & \geq 1
\end{matrix} 
\end{equation*}
Working mod $\pi$ we compute
\begin{align*}
  f_1 &= -B_{31} B_{22} B_{13} y^3 - C_{12} C_{21} A_{33} x z^2 
           + (\,\,\cdots) y^2 z\\
  f_2 &= A_{33} z (B_{31} C_{12} x z - C_{21} y ( B_{22}y + B_{32} z)) \\
  f_3 &= A_{33} z (B_{13} C_{21} x z - C_{12} y ( B_{22}y + B_{23} z))
\end{align*}
We see that $f_1,f_2,f_3$ each define a curve with a unique singular
point at $(1:0:0)$. If we multiply $B$, $C$, the last two rows and the
last two columns by $\pi$, then the cube is divisible by $\pi^2$. From
this we see that whichever two of the $f_i$ are chosen by the
procedure in (ii), the level is preserved and we are reduced to
Case~3.
\end{ProofOf}

\section{$2\times2\times2\times2$ hypercubes}
\label{sec:hypercubes}

We consider polynomials in $x_1,x_2, y_1, y_2, z_1, z_2, t_1,t_2$ that
are linear in each of the four sets of variables.  Such a polynomial
may be represented as
\begin{equation}
  \label{Hpoly}
\sum_{1\leq i,j,k,l \leq 2} H_{ijkl} x_i y_j z_k t_l
\end{equation}
where $H = (H_{ijkl})$ is a $2 \times 2 \times 2 \times 2$ hypercube.
A {\em hypercube} $H$ may be partitioned into two $2\times2\times2$
cubes in four distinct ways:
\begin{enumerate}
        \item $A_1 = (H_{1jkl})$ and $B_1=(H_{2jkl})$ 
        \item $A_2 = (H_{i1kl})$ and $B_2=(H_{i2kl})$ 
        \item $A_3 = (H_{ij1l})$ and $B_3=(H_{ij2l})$ 
        \item $A_4 = (H_{ijk1})$ and $B_4=(H_{ijk2})$ 
\end{enumerate}
Let $R$ be a ring. For each $1 \le i \le 4$ there is an action of
$\GL_2(R)$ on the space of hypercubes over $R$ via
\[ \left( \begin{matrix} r & s \\ t & u \end{matrix}\right): (A_i,B_i)
\mapsto (rA_i+sB_i,tA_i+uB_i). \] These actions commute, and so give
an action of $\GL_2(R)^4$.  We say that hypercubes are {\em
  $R$-equivalent} if they belong to the same orbit for this action.

For each $1 \le i < j \le 4$ there is an associated $(2,2)$-form
$F_{ij}$.  Indeed if we view~\eqref{Hpoly} as a bilinear form in $z_k$
and $t_l$, then the determinant of this form is a $(2,2)$-form in
$x_i$ and $y_j$:
\[
F_{12}=(\sum_{1 \leq i,j \leq 2}H_{ij11}x_{i}y_{j})(\sum_{1 \leq i,j
  \leq 2}H_{ij22}x_{i}y_{j})-(\sum_{1 \leq i,j \leq
  2}H_{ij12}x_{i}y_{j})(\sum_{1 \leq i,j \leq 2}H_{ij21}x_{i}y_{j}).
\]
The other $F_{ij}$ are defined similarly.
If $[M_1,M_2,M_3,M_4] \cdot H = H'$ then the $(2,2)$-forms
are related by \[[\det(M_3)\det(M_4),M_1,M_2] \cdot F_{12} = F'_{12}.\]
As seen in Section~\ref{sec:22}, each $(2,2)$-form determines a pair
of binary quartics. It turns out that the binary quartics in
$x_1,x_2$ associated to $F_{12},F_{13},F_{14}$ are all equal.
Thus a hypercube $H$ determines four binary quartics
$G_1, \ldots, G_4$, one in each of the four sets of variables.  Each
of these binary quartics has the same invariants $I$ and $J$.
Therefore the six $(2,2)$-forms $F_{ij}$ all have the same invariants
$c_4$, $c_6$ and $\Delta$.  We define $c_4(H)=c_4(F_{ij})$,
$c_6(H)=c_6(F_{ij})$ and $\Delta(H)=\Delta(F_{ij})$.

If $H$ is defined over a field and $\Delta(H) \not= 0$ then each of
the $F_{ij}$ defines a genus one curve in $\PP^1 \times \PP^1$.  These
curves are isomorphic, although not in a canonical way. 
(See \cite[Section 2.3]{BH} for further details.)  We write
$\CC_H$ to denote any one of them.

Ley $u$ and $v$ be the invariants in Lemma~\ref{lem:inv22}.  We find
that $u(F_{12}) = u(F_{34})$ and $v(F_{12}) = v(F_{34})$.  Therefore
$F_{12}$ and $F_{34}$ determine isomorphic pairs $(E,P)$.  (A~further
calculation is needed to check this in characteristics $2$ and $3$,
but we omit the details.)  Repeating for the other $F_{ij}$ gives a
tuple $(E,P_1,P_2,P_3)$ where $E$ is an elliptic curve and $0_E \not=
P_1,P_2,P_3 \in E$ with $P_1 + P_2 + P_3 = 0_E$.

We say that a hypercube $H$ is {\em integral} if it has coefficients
in $\OK$, and {\em non-singular} if $\Delta(H) \not= 0$.
\begin{lemma}
\label{lem:lev2222}
Let $H$ be a non-singular integral hypercube. Let $(E,P_1,P_2,P_3)$ be
the tuple determined by $H$. Then
\[ v(\Delta(H)) = v(\Delta_E) + 12
\max(\kappa(P_1),\kappa(P_2),\kappa(P_3)) + 12 \ell(H) \] where
$\ell(H) \ge 0$ is an integer we call the {\em level}.
\end{lemma}
\begin{proof}
This is immediate from Lemma~\ref{lem:lev22}.
\end{proof}

An integral hypercube is {\em saturated} if for all $1 \le i \le 4$
the cubes $A_i$ and $B_i$ are linearly independent mod $\pi$. If an
integral hypercube is not saturated, then it is obvious how we may
decrease the level.

Our algorithm for minimising hypercubes is described by the following
theorem.

\begin{thm} 
\label{thmH}
Let $H$ be a saturated hypercube with
associated $(2,2)$-forms $F_{ij}$. Suppose
that all of the $F_{ij}$ are non-minimal. 
Then by an $\OK$-equivalence, and permuting the sets of variables, 
we are in one of the following two situations:
\begin{enumerate}
\item The reduction of $F_{12}$ mod $\pi$ defines a curve in $\PP^1
  \times \PP^1$ with a unique singular point at $((1:0),(1:0))$, and
  the transformation
\begin{equation}
\label{trans12}
\left[ \frac{1}{\pi} \begin{pmatrix} 1 & 0 \\ 0 & \pi \end{pmatrix},
  \begin{pmatrix} 1 & 0 \\ 0 & \pi \end{pmatrix},
  \begin{pmatrix} 1 & 0 \\ 0 & 1 \end{pmatrix},
  \begin{pmatrix} 1 & 0 \\ 0 & 1 \end{pmatrix} \right]
\end{equation}
gives an integral hypercube of the same level.
\item We have $F_{12} \equiv x_2^2 y_2^2 \pmod{\pi}$ and 
  the transformation~\eqref{trans12} gives a non-saturated
  hypercube of the same level.
\end{enumerate}
Moreover, at most two iterations of the procedure in (i) are needed
to give a non-saturated hypercube, or to reach the situation in (ii).
\end{thm}

We initially used the methods in Sections~\ref{sec:22}
and~\ref{sec:cubes} to prove Theorem~\ref{thmH} under the 
hypothesis that $H$ is non-minimal. The advantage of the theorem as
stated here is that it has the following consequence.

\begin{Corollary}
\label{cor}
Let $H$ be a integral hypercube with associated $(2,2)$-forms
$F_{ij}$. Then $H$ is minimal if and only if some $F_{ij}$ is minimal.
\end{Corollary}

\begin{Remark}
We may represent $H = (H_{ijkl})$ as a $4\times4$ matrix:
\begin{equation}
\label{Hmat}
\left(
\begin{array}{cc|cc}
  H_{1111} &  H_{1211} & H_{1112}&  H_{1212}\\
 H_{2111} & H_{2211} & H_{2112} &  H_{2212}\\ \hline
  H_{1121} &  H_{1221} & H_{1122}&  H_{1222}\\
 H_{2121} & H_{2221} & H_{2122} &  H_{2222}\\
\end{array} \right).
\end{equation}
If we write $r_1,r_2,r_3,r_4$ for the rows, then the first copy of
$\GL_2$ acts by row operations simultaneously on $\{r_1,r_2\}$ and
$\{r_3,r_4\}$, the third copy of $\GL_2$ acts by row operations on
$\{r_1,r_3\}$ and $\{r_2,r_4\}$, and the other two copies of $\GL_2$
act by column operations.
\end{Remark}

\begin{Remark}
  Let $H$ be an integral hypercube with associated binary quartics
  $G_1, \ldots, G_4$. It is clear that if any of the $G_i$ are minimal
  then $H$ is minimal. However the converse is not true. For example
  if
 \begin{equation*}
\label{Hmat:ce}
H \equiv 
\left( \begin{array}{cc|cc}
    1 & 0 & 0 &  0 \\
    0 &  1 & 0 & 0 \\ \hline
    0 &  0 & 1 & 0\\
    0 &  0 & 0 & 1
\end{array} \right) \pmod{\pi^2}
\end{equation*}
then $H$ is minimal (since $F_{12} \equiv (x_1 y_1 + x_2 y_2)^2 \pmod{\pi^2}$ 
and we saw in Remark~\ref{remFG} that this is minimal), 
yet we have $G_1 \equiv \ldots \equiv G_4 \equiv 0 \pmod{\pi^2}$.
\end{Remark}

For the proof of Theorem~\ref{thmH} we need the following lemma.
\begin{lemma}
\label{lemA}
Let $H$ be an integral hypercube.  Suppose that at least one of the
associated $(2,2)$-forms $F_{ij}$ is non-minimal.  Then by an
$\OK$-equivalence, and permuting the sets of variables, we may assume
$H_{11kl} \equiv 0 \pmod{\pi}$ for all $1 \le k,l \le 2$.
\end{lemma}
\begin{proof} 
  We suppose that $F_{12}$ is non-minimal.  If the reduction of
  $F_{12}$ mod $\pi$ is non-zero, then by Theorem~\ref{thm:min22} it
  defines a curve in $\PP^1 \times \PP^1$ with singular locus a point,
  a line or a pair of lines.  We may assume by an $\OK$-equivalence
  that the curve is singular at $((1:0),(1:0))$. If $H_{11kl}
  \not\equiv 0 \pmod{\pi}$ for some $1 \le k,l \le 2$ then we may
  assume by an $\OK$-equivalence that $H_{1111} \not\equiv 0
  \pmod{\pi}$.  A further $\OK$-equivalence gives
  \[ H_{2111} \equiv H_{1211} \equiv H_{1121} \equiv H_{1112} \equiv 0
  \pmod{\pi}. \] Since the coefficients of $x_1^2 y_1^2$, $x_1^2 y_1
  y_2$ and $x_1 x_2 y_1^2$ in $F_{12}$ vanish mod $\pi$, we have
  \[ H_{1122} \equiv H_{1222} \equiv H_{2122} \equiv 0 \pmod{\pi}. \]
  Lemma~\ref{lemB}(i) now shows that either
  \[ H_{1221} H_{1212} \equiv 0 \pmod{\pi} \quad \text{ or } \quad
  H_{2121} H_{2112} \equiv 0 \pmod{\pi}. \] By switching the first two
  sets of variables and switching the last two sets of variables, as
  necessary, we may assume $H_{1212} \equiv 0 \pmod{\pi}$. Now
  $H_{1jk2} \equiv 0 \pmod{\pi}$ for all $1 \le j,k \le 2$, and this
  proves the lemma.
\end{proof}

\medskip

\begin{ProofOf}{Theorem~\ref{thmH}} 
By Lemma~\ref{lemA} we may assume $H_{11kl} \equiv 0 \pmod{\pi}$ 
for all $1 \le k,l \le 2$. Applying Lemma~\ref{lemB}(i) to $F_{12}$,
and switching the first two sets of variables if necessary, we have
\[ H_{1211} H_{1222} - H_{1212} H_{1221} \equiv 0 \pmod{\pi}. \] By an
$\OK$-equivalence we may assume $H_{1jkl} \equiv 0 \pmod{\pi}$ for all
$1 \le j,k,l \le 2$, except $(j,k,l) = (2,2,2)$.  Since $H$ is
saturated we have $H_{1222} \not\equiv 0 \pmod{\pi}$.  Again by
Lemma~\ref{lemB}(i) we have $H_{2111} \equiv 0 \pmod{\pi}$.

We now split into cases, according as to whether 
\begin{equation}
\label{star}
  H_{2211} \equiv H_{2121} \equiv H_{2112} \equiv 0 \pmod{\pi}. 
\end{equation} 

If this condition is not satisfied, then by permuting the last three
sets of variables, we may suppose $H_{2211} \not\equiv 0
\pmod{\pi}$. By an $\OK$-equivalence we have
 \begin{equation}
\label{Hmat:case1}
H \equiv 
\left( \begin{array}{cc|cc}
    0 &  0 & 0 & 0 \\
    0 &  1 & \beta & 0 \\ \hline
    0 &  0 & 0 & 1 \\
    \alpha &  0 & \gamma & 0
  \end{array} \right) \pmod{\pi}
\end{equation}
for some $\alpha,\beta,\gamma \in k$.  We compute $F_{12} \equiv x_1
x_2 y_2^2 + x_2^2( \alpha \beta y_1^2 + \gamma y_1 y_2)
\pmod{\pi}$. The conclusions in (i) are satisfied unless $\alpha \beta
= \gamma = 0$. In the remaining case we may assume, by switching the
last two sets of variables if necessary, that $\alpha = 0$.  Now
switching the first and last sets of variables, and swapping over the
third set of variables (i.e. $z_1 \leftrightarrow z_2$), we may swap
over $\beta$ and $\gamma$. Therefore $\beta = \gamma = 0$, and this
contradicts that $H$ is saturated.

Now suppose the condition~\eqref{star} is satisfied. Then by an
$\OK$-equivalence (and our assumption that $H$ is saturated) we have
\begin{equation}
\label{Hmat:case2}
H  \equiv  
\left( \begin{array}{cc|cc}
    0 & 0 & 0 & 0 \\
    0 &  0 & 0 & 1 \\ \hline
    0 &  0 & 0 & 1\\
    0 &  -1 & -1 &  0
  \end{array} \right) \pmod{\pi}.
\end{equation}
We compute $F_{12} \equiv x_2^2 y_2^2 \pmod{\pi}$.  Let $F_{12}$ have
coefficients $a_{ij}$ as labelled in~\eqref{mat:22}.
Lemma~\ref{lemB}(ii) shows that either $v(a_{12}) \ge 2$ or $v(a_{21})
\ge 2$. Therefore $v(H_{1111}) \ge 2$.  Again by Lemma~\ref{lemB}(ii)
we have either $v(a_{11}) \ge 3$, $v(a_{13}) \ge 2$ or $v(a_{31}) \ge
2$. Therefore at least one of the coefficients $H_{2111}$, $H_{1211}$,
$H_{1121}$, $H_{1112}$ has valuation at least two. By permuting the
sets of variables we may suppose $v(H_{1112}) \ge 2$.  The conclusions
in (ii) are now satisfied.

To prove the last part of the theorem, we need the following lemma.

\begin{lemma}
\label{viaBQ}
Let $H$ be a hypercube over a field $k$ with associated $(2,2)$-forms
$F_{ij}$. We write
\begin{align*}
F_{12} &= f_1(x_1,x_2) y_1^2 + f_2(x_1,x_2) y_1 y_2 + f_3(x_1,x_2) y_2^2  \\
F_{13} &= g_1(x_1,x_2) z_1^2 + g_2(x_1,x_2) 
   z_1 z_2 + g_3(x_1,x_2) z_2^2 
\end{align*}
\begin{enumerate}
\item We have $g_2 = f_2 + 2 h$ and $g_1 g_3 = f_1 f_3 + f_2 h + h^2$
  for some $h \in k[x_1,x_2]$.
\item If $f_1 = f_2 = 0$ and $g_1$, $g_2$ are multiples of $x_2^2$,
  then $F_{13}$ is either zero or factors as a product of binary
  quadratic forms.
\end{enumerate}
\end{lemma}
\begin{proof} (i) We have already remarked that $f_2^2 - 4 f_1 f_3 =
  g_2^2 - 4 g_1 g_3$.  The result follows by considering the $f_i$ and
  $g_i$ as polynomials in $\Z[H_{ijkl}][x_1,x_2]$.

  (ii) By (i) we have $g_1 = \alpha x_2^2$, $g_2 = 2 \beta x_2^2$ and
  $\alpha x_2^2 g_3(x_1,x_2) = \beta x_2^4$. If $\alpha = 0$ then $g_1
  = g_2 = 0$, whereas if $\alpha \not=0$ then $g_1$, $g_2$, $g_3$ are
  multiples of $x_2^2$.
\end{proof}

We say that a $(2,2)$-form $F$ is {\em slender} if $F$ mod $\pi$ is
either zero, or factors as a product of binary quadratic forms.
Theorem~\ref{thm:min22} shows that if $F$ is non-minimal then either
$F$ mod $\pi$ defines a curve with a unique singular point, or $F$ is
slender. These possibilities are mutually exclusive.

We now complete the proof of Theorem~\ref{thmH}.  Applying the
transformation in (i) to $H$ has the effect of applying the
transformation in Theorem~\ref{thm:min22}(iii) to $F_{12}$.  The last
sentence of Theorem~\ref{thm:min22} tells us that, after applying this
transformation, $F_{12}$ mod $\pi$ is either zero, or factors as a
product of binary quadratic forms both of which have a repeated root.
In particular $F_{12}$ is slender.

We claim that $F_{13}$ is slender. If not then $F_{13}$ mod $\pi$
defines a curve with a unique singular point. By an $\OK$-equivalence
we may assume that this point is $((1:0),(1:0))$, and that $F_{12}
\equiv f_3(x_1,x_2) y_2^2 \pmod{\pi}$ for some binary quadratic form
$f_3$. Lemmas~\ref{lemB}(i) and~\ref{viaBQ}(ii) now show that $F_{13}$
is slender.

The same argument shows that all of the $F_{ij}$ are slender, except
possibly $F_{34}$. Since $F_{34}$ was unchanged by the
transformation~\eqref{trans12}, it follows that after at most two
iterations, all of the $F_{ij}$ are slender. In particular we cannot
return to the situation in (i), and this completes the proof.
\end{ProofOf}

\section{Minimisation Theorems}
\label{sec:minthm}

The algorithms in \cite{CFS} and~\cite{F5} for minimising genus one
curves of degree $2,3,4,5$ were complemented by a more theoretical
result. This stated that if a genus one curve is soluble over $K$ (or
more generally over an unramified extension) then the discriminant of
a minimal model is the same as that for the Jacobian elliptic curve.
In this section we prove the analogue of this result for
$(2,2)$-forms, $3 \times 3 \times 3$ cubes and $2 \times 2 \times 2
\times 2$ hypercubes.

In earlier papers, most notably \cite[Lemmas 3,4,5]{BSD}, the
minimisation algorithms and minimisation theorems were treated
together.  Following~\cite{CFS} we separate these out, and this leads
to clean results that work the same in all residue characteristics.
We phrase our result in terms of the level, as defined in
Lemmas~\ref{lem:lev22}, \ref{lem:lev333} and~\ref{lem:lev2222}.
\begin{thm}
\label{minthm}
Let $\Phi$ be a nonsingular $(2,2)$-form, $3 \times 3 \times 3$ cube,
or $2 \times 2 \times 2 \times 2$ hypercube defined over $K$.  If
$\CC_\Phi(K) \not= \emptyset$ then $\Phi$ has minimal level $0$.
\end{thm}

\begin{Remark} \label{remunram} The algorithms in
  Sections~\ref{sec:22},~\ref{sec:cubes} and~\ref{sec:hypercubes} show
  that the minimal level is unchanged by an unramified field
  extension.  The hypothesis in Theorem~\ref{minthm} may therefore be
  weakened to solubility over an unramified field extension.  We give 
  examples below to show that this hypothesis cannot be removed
  entirely.
\end{Remark}

Let $E/K$ be an elliptic curve and $n \in \{2,3\}$.  Let $D$ and $D'$
be $K$-rational divisors on $E$ of degree $n$.  The image of $E$ in
$\PP^{n-1} \times \PP^{n-1}$ via $|D| \times |D'|$ is defined by a
$(2,2)$-form in the case $n=2$, and three bilinear forms in the case
$n=3$.  The coefficients of the latter give a $3 \times 3 \times 3$
cube.  We note that the $(2,2)$-form, respectively $3 \times 3 \times
3$ cube, is uniquely determined up to $K$-equivalence by the triple
$(E,[D],[D'])$, where $[D]$ denotes the linear equivalence class of
$D$.  Moreover every $(2,2)$-form, respectively $3 \times 3 \times 3$
cube, defining a non-singular genus one curve with a $K$-rational
point, arises in this way. Therefore the first two cases of
Theorem~\ref{minthm} are immediate from the following theorem.

We write $\operatorname{sum} : \Div_K(E) \to E(K)$ for the map that
sends a formal sum of points to its sum using the group law on $E$.

\begin{thm}
\label{minEDD}
Let $E/K$ be an elliptic curve with integral Weierstrass equation
\begin{equation}
\label{Weqn1}
y^2 + a_1 x y + a_3 y = x^3 + a_2 x^2 + a_4 x 
\end{equation}
and let $P= (0,0) \in E(K)$. Let $D, D' \in \Div_K(E)$ be divisors of
degree $n \in \{2,3\}$ with $\operatorname{sum} (D' - D) = P$.  Then
$(E,[D],[D'])$ may be represented by an integral $(2,2)$-form, or $3
\times 3 \times 3$ cube, with the same discriminant as~(\ref{Weqn1}).
\end{thm}

We start by proving Theorem~\ref{minEDD} in the case $D \sim n . 0_E$.
Since $\operatorname{sum} (D' - D) = P$ we have $D' \sim (n-1).0_E +
P$. We put
\[f = \frac{y + a_1 x + a_3}{x} = \frac{x^2 + a_2 x + a_4}{y}\] and
split into the cases $n=2$ and $n=3$.

\medskip

\paragraph{{\bf Case $n=2$}} The embedding $E \to \PP^1 \times \PP^1$
via $|D| \times |D'|$ is given by
\[ (x,y) \mapsto ((1:x),(1:f)).\] The image is defined by the
$(2,2)$-form
\[ F(x_1,x_2; y_1,y_2) =  x_2^2 y_1^2 - x_1 x_2 y_2^2 
  + x_1 y_1 (a_1 x_2 y_2 + a_2 x_2 y_1 + a_3 x_1 y_2 + a_4 x_1 y_1),
\]
with the same discriminant as~\eqref{Weqn1}.

\medskip

\paragraph{{\bf Case $n=3$}} The embedding $E \to \PP^2 \times \PP^2$
via $|D| \times |D'|$ is given by
\[ (x,y) \mapsto ((1:x:y),(1:x:f)). \] The image is defined by
bilinear forms
\begin{align*}
 B_1(x_1,x_2,x_3; y_1,y_2,y_3) &=  x_2 y_1 - x_1 y_2, \\
 B_2(x_1,x_2,x_3; y_1,y_2,y_3) &=  
   x_3 y_1  + a_1 x_2 y_1 + a_3 x_1 y_1 - x_2 y_3, \\
 B_3(x_1,x_2,x_3; y_1,y_2,y_3) &=  
   x_2 y_2 + a_2 x_2 y_1 + a_4 x_1 y_1 - x_3 y_3.
\end{align*}
The coefficients of $B_1, B_2, B_3$ give a $3 \times 3 \times 3$ cube,
and this has the same discriminant as~\eqref{Weqn1}.

\begin{lemma}
\label{lem:3->2}
Let $S$ be a $3 \times 3 \times 3$ cube corresponding to bilinear
forms $B_1, B_2, B_3$, defining $C\subset \PP^2 \times \PP^2$ a smooth
curve of genus one, embedded via $|D| \times |D'|$.
\begin{enumerate}
\item If $Q = ((0:0:1),(0:0:1)) \in C(K)$ then for $i=1,2,3$ we can
write
\[ B_i = L_i(y_1,y_2) x_3 + M_i(x_1,x_2) y_3 +
N_i(x_1,x_2;y_1,y_2). \]
\item The image of $C$ in $\PP^1 \times \PP^1$ via $|D-Q| \times
  |D'-Q|$ is defined by the $(2,2)$-form
  \[ F(x_1,x_2;y_1,y_2) = \left| \begin{matrix}
      L_1 & M_1 & N_1 \\
      L_2 & M_2 & N_2 \\
      L_3 & M_3 & N_3
    \end{matrix} \right|. \]
\item We have $\Delta(F) = \Delta(S)$.
\end{enumerate}
\end{lemma}
\begin{proof}
  We map $C \to \PP^1 \times \PP^1$ via $((x_1:x_2),(y_1:y_2))$.  The
  first two statements are clear. For (iii) we checked by a generic
  calculation that $F$ and $S$ have the same invariants $c_4$ and
  $c_6$.
\end{proof}

\begin{lemma}
\label{lem:2->3}
Let $F$ be a $(2,2)$-form defining $C \subset \PP^1 \times \PP^1$ a
smooth curve of genus one, embedded via $|D| \times |D'|$.

\begin{enumerate}

\item If $Q= ((1:0),(1:0)) \in C(K)$ then we can write
\begin{equation*}
F(x_1,x_2;y_1,y_2) = \begin{pmatrix} x_1^2 & x_1 x_2 & x_2^2 \end{pmatrix}
\begin{pmatrix} 
0 & a_{12} & a_{13} \\
a_{21} & a_{22} & a_{23} \\
a_{31} & a_{32} & a_{33} 
\end{pmatrix} \begin{pmatrix} y_1^2 \\ y_1y_2 \\ y_2^2 \end{pmatrix}.
\end{equation*}
\item The image of $C$ in $\PP^2 \times \PP^2$ via $|D+Q| \times
  |D'+Q|$ is defined by the $3 \times 3 \times 3$ cube $S$ with
  entries
 \begin{equation*}
    \begin{pmatrix}
      0 &  1 & 0 \\
      1 & a_{22} & a_{23} \\
      0 & a_{32} & a_{33}
\end{pmatrix} \hspace{2em}
\begin{pmatrix}
  0 & 0 & 0 \\
  0 & a_{12} &  a_{13}  \\
  -1 & 0 & 0
\end{pmatrix} \hspace{2em}
\begin{pmatrix}
  0 & 0 &  -1 \\
  0 & a_{21} & 0 \\
  0 & a_{31} & 0
\end{pmatrix}. 
\end{equation*}
\item We have $\Delta(S) = \Delta(F)$.
\end{enumerate}
\end{lemma}
\begin{proof}
  We have $D \sim Q + R$ and $D' \sim Q+ R'$ where $R=
  ((1:0),(-a_{13}:a_{12}))$ and $R' = ((-a_{31}:a_{21}),(1:0))$.
  Choosing bases for the space of bilinear forms vanishing at $R'$,
  and the space of bilinear forms vanishing at $R$, we find that the
  map $C \to \PP^2 \times \PP^2$ via $|D+Q| \times |D'+Q|$ is given by
\begin{align*}
  \big((x_1:x_2),(y_1:y_2)\big) \mapsto \big(
((a_{21} x_1 & + a_{31} x_2) y_1 : x_1 y_2  : x_2 y_2 ), \\
& (x_1 (a_{12} y_1 + a_{13} y_2) : x_2 y_1 : x_2 y_2 )\big). 
\end{align*}
The image is defined by
\begin{align*}
B_1 &= x_2 y_1 + x_1 y_2 + a_{22} x_2 y_2 + a_{32} x_3 y_2 + a_{23} x_2 y_3 + a_{33} x_3 y_3 \\
B_2 &= -x_3 y_1 + a_{12} x_2 y_2 + a_{13} x_2 y_3, \\
B_3 &= -x_1 y_3 + a_{21} x_2 y_2 + a_{31} x_3 y_2.
\end{align*}
The coefficients of these forms give the cube $S$ in the statement of
the lemma. Again we prove (iii) by a generic calculation.
\end{proof}

\begin{ProofOf}{Theorem~\ref{minEDD}}
We split into the cases $n=2$ and $n=3$.

\medskip

\paragraph{{\bf Case $n=2$}}
We have $D \sim 3.0_E - Q$ for some $Q \in E(K)$. By the special case
of the theorem already established, there is an integral $3 \times 3
\times 3$ cube representing $(E,[D+Q],[D'+Q])$, with the same
discriminant as~\eqref{Weqn1}. We have $E \subset \PP^2 \times
\PP^2$. Since $\SL_3(\OK)$ acts transitively on $\PP^2(K)$ we may
assume $Q = ((0:0:1),(0:0:1))$. Then Lemma~\ref{lem:3->2} give an
integral $(2,2)$-form representing $(E,[D],[D'])$, with the same
discriminant as~\eqref{Weqn1}.

\medskip

\paragraph{{\bf Case $n=3$}}
We have $D \sim 2.0_E + Q$ for some $Q \in E(K)$. By the special case
of the theorem already established, there is an integral $(2,2)$-form
representing $(E,[D-Q],[D'-Q])$, with the same discriminant
as~\eqref{Weqn1}. We have $E \subset \PP^1 \times \PP^1$. Since
$\SL_2(\OK)$ acts transitively on $\PP^1(K)$ we may assume $Q =
((1:0),(1:0))$. Then Lemma~\ref{lem:2->3} give an integral $3 \times 3
\times 3$ cube representing $(E,[D],[D'])$, with the same discriminant
as~\eqref{Weqn1}.
\end{ProofOf}

This completes the proof of Theorem~\ref{minthm} for $(2,2)$-forms and
$3 \times 3 \times 3$ cubes. We now deduce the result for hypercubes
from the result for $(2,2)$-forms.  Let $H$ be a non-singular
hypercube over $K$, with associated $(2,2)$-forms $F_{ij}$. The genus
one curve $\CC_H$ is that defined by any of the $F_{ij}$. So if
$\CC_H(K) \not=0$ then the result for $(2,2)$-forms shows that each
$F_{ij}$ has minimal level $0$.  By the definitions in
Lemmas~\ref{lem:lev22} and~\ref{lem:lev2222}, we have $\ell(H) = \min
\ell(F_{ij})$. It follows by Corollary~\ref{cor} that $H$ has minimal
level $0$.

\begin{Remark}
\label{remcrit}
We give some examples to show that the minimal level can be
positive. We assume for convenience that $\charic(k) \not=2,3$.  A
binary quartic, or ternary cubic is called critical (see \cite[Section
5]{CFS}) if the valuations of its coefficients satisfy
\[ 
=1 \quad \ge 2 \quad \ge 2 \quad \ge 3 \quad =3 \qquad \text{ or }
\qquad
\begin{array}{cccccccc}
  & & & \multicolumn{2}{c}{ = 2 } \\
  & & \multicolumn{2}{c}{ \ge 2 } & \multicolumn{2}{c}{ \ge 2 } \\
  & \multicolumn{2}{c}{ \ge 1 } & \multicolumn{2}{c}{ \ge 1 } & 
  \multicolumn{2}{c}{ \ge 2 } \\
  \multicolumn{2}{c}{ =0 } & \multicolumn{2}{c}{ \ge 1 } & 
  \multicolumn{2}{c}{ \ge 1 } & \multicolumn{2}{c}{ = 1 }
\end{array} 
\]
We now define a {\em critical} $(2,2)$-form, $3 \times 3 \times 3$ cube or
$2 \times 2 \times 2 \times 2$ hypercube, to be one whose coefficients
have valuations satisfying
\[
\begin{matrix}
  = 2 & \geq 2 & = 1 \\
  \geq 2 & \geq 1 & \geq 1 \\
  = 1 & \geq 1 & = 0
\end{matrix}
\]
or
 \begin{equation*}
    \begin{matrix}
      \geq 2 &    = 1 & \geq 1 \\
      = 1 & \geq 1 & \geq 1 \\
      \geq 1 & \geq 1 & = 0
\end{matrix} \hspace{2em}
\begin{matrix}
  = 1 & \geq 1 & \geq 1 \\
  \geq 1 & \geq 1 &    = 0  \\
  \geq 1 & = 0 & \geq 0
\end{matrix} \hspace{2em}
\begin{matrix}
  \geq 1 & \geq 1 &    = 0 \\
  \geq 1 &    = 0 & \geq 0 \\
  = 0 & \geq 0 & \geq 0
\end{matrix} 
\end{equation*}
or
\[
\begin{array}{cc|cc}
  \geq 2 &     = 1 &    = 1 & \geq 1 \\
  = 1 &  \geq 1 & \geq 1 &  = 0 \\ \hline
  = 1 &  \geq 1 & \geq 1 &  = 0 \\
  \geq 1 &     = 0 &    = 0 & \geq 0 
\end{array} 
\]
Either by using our algorithms, or observing that the corresponding
binary quartics and ternary cubics are critical, we see that any such
model $\Phi$ is minimal. However by applying the transformation
\[ [\pi^{-2} ,A_2,A_2], [\pi^{-4/3} A_3,A_3,A_3] \text{ or }
[\pi^{-3/2} A_2, A_2,A_2,A_2], \] where
\[ A_2 =  \begin{pmatrix} 1 & \\
  & \pi^{1/2} \end{pmatrix} \text{ and }
A_3 = \begin{pmatrix} 1 & & \\
  & \pi^{1/3} & \\ & & \pi^{2/3} \end{pmatrix}, \] we see that
$I(\Phi) \equiv 0 \pmod{\pi^p}$ for any invariant $I$ of weight
$p$. Therefore $\Phi$ has positive level.
\end{Remark}


\begin{thebibliography}{9}

\bibitem{BH} 
M. Bhargava and W. Ho, 
Coregular spaces and genus one curves, {\em Camb. J. Math.} {\bf{4}}
(2016), no. 1, 1--119.

\bibitem{BSD} 
B.J. Birch and H.P.F. Swinnerton-Dyer, 
Notes on elliptic curves I, {\em J. reine angew. Math.}  {\bf{212}}
(1963) 7--25.

\bibitem{CFS} 
J.E. Cremona, T.A. Fisher and M. Stoll, 
Minimisation and reduction of 2-, 3- and 4-coverings of elliptic
curves {\em Algebra \& Number Theory} \textbf{4} (2010), no. 6,
763--820

\bibitem{F5}
T.A. Fisher, 
Minimisation and reduction of 5-coverings of elliptic curves, {\em
  Algebra \& Number Theory} \textbf{7} (2013), no. 5, 1179--1205.

\bibitem{CTP3}
T.A. Fisher and R.D. Newton, 
Computing the Cassels-Tate pairing on the 3-Selmer group of an
elliptic curve, {\em Int. J. Number Theory} \textbf{10} (2014), no. 7,
1881--1907.

\bibitem{Sil}
J.H. Silverman, 
{\em The arithmetic of elliptic curves}, Graduate Texts in Mathematics
{\bf{106}}, Springer-Verlag, New York, 1986.

\end{thebibliography}
\end{document}